\documentclass{amsart}
\usepackage{amsfonts}

\newtheorem{theorem}{Theorem}
\theoremstyle{plain}

\newtheorem{corollary}{Corollary}

\newtheorem{lemma}{Lemma}

\newtheorem{proposition}{Proposition}
\newtheorem{remark}{Remark}

\numberwithin{equation}{section}
\input{tcilatex}

\begin{document}
\title[Noncommutative Perspectives of Operator Monotone Functions]{%
Noncommutative Perspectives of Operator Monotone Functions in Hilbert Spaces}
\author{Silvestru Sever Dragomir$^{1,2}$}
\address{$^{1}$Mathematics, College of Engineering \& Science\\
Victoria University, PO Box 14428\\
Melbourne City, MC 8001, Australia.}
\email{sever.dragomir@vu.edu.au}
\urladdr{http://rgmia.org/dragomir}
\address{$^{2}$School of Computer Science \& Applied Mathematics, University
of the Witwatersrand, Private Bag 3, Johannesburg 2050, South Africa}

\begin{abstract}
Assume that the function $f:(0,\infty )\rightarrow \mathbb{R}$ is operator
monotone in $(0,\infty )$ and has the representation 
\begin{equation*}
f\left( t\right) =a+bt+\int_{0}^{\infty }\frac{t\lambda }{t+\lambda }%
dw\left( \lambda \right) ,
\end{equation*}%
where $a\in \mathbb{R}$, $b\geq 0$ and $w$ is a positive measure on $%
(0,\infty )$. In this paper we obtained among others that 
\begin{align*}
& \mathcal{P}_{f}\left( B,P\right) -\mathcal{P}_{f}\left( A,P\right) \\
& =b\left( B-A\right) +\int_{0}^{\infty }\lambda ^{2}\left[
\int_{0}^{1}P\left( \left( 1-t\right) A+tB+\lambda P\right) ^{-1}\left(
B-A\right) \right. \\
& \left. \times \left( \left( 1-t\right) A+tB+\lambda P\right) ^{-1}Pdt 
\right] dw\left( \lambda \right)
\end{align*}%
for all $A,$ $B,$ $P>0.$ Applications for \textit{weighted operator
geometric mean} and \textit{relative operator entropy} are also provided.
\end{abstract}

\maketitle

\section{Introduction}

Consider a complex Hilbert space $\left( H,\left\langle \cdot ,\cdot
\right\rangle \right) $. An operator $T$ is said to be positive (denoted by $%
T\geq 0$) if $\left\langle Tx,x\right\rangle \geq 0$ for all $x\in H$ and
also an operator $T$ is said to be \textit{strictly positive} (denoted by $%
T>0$) if $T$ is positive and invertible. A real valued continuous function $%
f $ on $(0,\infty )$ is said to be operator monotone if $f(A)\geq f(B)$
holds for any $A\geq B>0.$

We have the following integral representation for the power function when $%
t>0,$ $r\in (0,1],$ see for instance \cite[p. 145]{Bh}%
\begin{equation}
t^{r-1}=\frac{\sin \left( r\pi \right) }{\pi }\int_{0}^{\infty }\frac{%
\lambda ^{r-1}}{\lambda +t}d\lambda .  \label{e.0}
\end{equation}

Observe that for $t>0,$ $t\neq 1,$ we have 
\begin{equation*}
\int_{0}^{u}\frac{d\lambda }{\left( \lambda +t\right) \left( \lambda
+1\right) }=\frac{\ln t}{t-1}+\frac{1}{1-t}\ln \left( \frac{u+t}{u+1}\right)
\end{equation*}%
for all $u>0.$

By taking the limit over $u\rightarrow \infty $ in this equality, we derive 
\begin{equation*}
\frac{\ln t}{t-1}=\int_{0}^{\infty }\frac{d\lambda }{\left( \lambda
+t\right) \left( \lambda +1\right) },
\end{equation*}%
which gives the representation for the logarithm%
\begin{equation}
\ln t=\left( t-1\right) \int_{0}^{\infty }\frac{d\lambda }{\left( \lambda
+1\right) \left( \lambda +t\right) }  \label{e.2}
\end{equation}%
for all $t>0.$

In 1934, K. L\"{o}wner \cite{L} had given a definitive characterization of
operator monotone functions as follows, see for instance \cite[p. 144-145]%
{Bh}:

\begin{theorem}
\label{t.A}A function $f:(0,\infty )\rightarrow \mathbb{R}$ is operator
monotone in $(0,\infty )$ if and only if it has the representation 
\begin{equation}
f\left( t\right) =a+bt+\int_{0}^{\infty }\frac{t\lambda }{t+\lambda }%
dw\left( \lambda \right)  \label{e.1}
\end{equation}%
where $a\in \mathbb{R}$ and $b\geq 0$ and a positive measure $w$ on $%
(0,\infty )$ such that 
\begin{equation*}
\int_{0}^{\infty }\frac{\lambda }{1+\lambda }dw\left( \lambda \right)
<\infty .
\end{equation*}
\end{theorem}

We recall the important fact proved by L\"{o}wner and Heinz that states that
the power function $f:(0,\infty )\rightarrow \mathbb{R}$, $f\left( t\right)
=t^{\alpha }$ is an operator monotone function for any $\alpha \in \left[ 0,1%
\right] ,$ \cite{H}. The function $\ln $ is also operator monotone on $%
\left( 0,\infty \right) .$

For other examples of operator monotone functions, see \cite{FS} and \cite%
{Fu1}.

Let $f$ be a continuous function defined on the interval $I$ of real
numbers, $B$ a selfadjoint operator on the Hilbert space $H$ and $A$ a
positive invertible operator on $H.$ Assume that the spectrum $\limfunc{Sp}%
\left( A^{-1/2}BA^{-1/2}\right) \subset \mathring{I}.$ Then by using the
continuous functional calculus, we can define the \textit{perspective} $%
\mathcal{P}_{f}\left( B,A\right) $ by setting 
\begin{equation*}
\mathcal{P}_{f}\left( B,A\right) :=A^{1/2}f\left( A^{-1/2}BA^{-1/2}\right)
A^{1/2}.
\end{equation*}%
If $A$ and $B$ are commutative, then 
\begin{equation*}
\mathcal{P}_{f}\left( B,A\right) =Af\left( BA^{-1}\right)
\end{equation*}%
provided $\limfunc{Sp}\left( BA^{-1}\right) \subset \mathring{I}.$

For any function $f:\left( 0,\infty \right) \rightarrow \mathbb{R}$ the
transpose $\tilde{f}$ of $f$ is defined by 
\begin{equation*}
\tilde{f}\left( x\right) =xf\left( x^{-1}\right) ,\text{ }x>0.
\end{equation*}

It is well known that (see for instance \cite{NS}), if $f:\left( 0,\infty
\right) \rightarrow \mathbb{R}$ is continuous on $\left( 0,\infty \right) ,$
then for all $A,$ $B>0,$%
\begin{equation}
\mathcal{P}_{\tilde{f}}\left( A,B\right) =\mathcal{P}_{f}\left( B,A\right) .
\label{e.3.1}
\end{equation}

If $f$ is nonnegative and operator monotone on $\left( 0,\infty \right) $,
then $\tilde{f}$ is operator monotone on $\left( 0,\infty \right) ,$ see 
\cite{NS}.

The following inequality is of interest, see \cite{NS}:

\begin{theorem}
\label{t.C}Assume that $f$ is nonnegative and operator monotone on $\left(
0,\infty \right) .$ If $A\geq C>0$ and $B\geq D>0,$ then%
\begin{equation}
\mathcal{P}_{f}\left( A,B\right) \geq \mathcal{P}_{f}\left( C,D\right) .
\label{e.3.2}
\end{equation}
\end{theorem}

It is well known that (see \cite{E} and \cite{ENG} or \cite{EH}), if $f$ is
an \textit{operator convex function} defined in the positive half-line, then
the mapping 
\begin{equation*}
\left( B,A\right) \mapsto \mathcal{P}_{f}\left( B,A\right)
\end{equation*}%
\textit{defined in pairs of positive definite operators, is operator convex.}

If $f_{\nu }:[0,\infty )\rightarrow \lbrack 0,\infty ),$ $f_{\nu }\left(
t\right) =t^{\nu },$ $\nu \in \left[ 0,1\right] ,$ then%
\begin{equation*}
P_{f_{\nu }}\left( B,A\right) :=A^{1/2}\left( A^{-1/2}BA^{-1/2}\right) ^{\nu
}A^{1/2}=:A\sharp _{\nu }B,
\end{equation*}%
is the \textit{weighted operator geometric mean }of the positive invertible
operators $A$ and $B$ with the weight $\nu .$

We define the \textit{weighted operator arithmetic mean }by%
\begin{equation*}
A\nabla _{\nu }B:=\left( 1-\nu \right) A+\nu B,\text{ }\nu \in \left[ 0,1%
\right] .
\end{equation*}%
It is well known that the following \textit{Young's type inequality} holds:%
\begin{equation*}
A\sharp _{\nu }B\leq A\nabla _{\nu }B
\end{equation*}%
for any $\nu \in \left[ 0,1\right] .$

If we take the function $f=\ln ,$ then 
\begin{equation*}
P_{\ln }\left( B,A\right) :=A^{1/2}\ln \left( A^{-1/2}BA^{-1/2}\right)
A^{1/2}=:S\left( A|B\right) ,
\end{equation*}%
is the \textit{relative operator entropy}, for positive invertible operators 
$A$ and $B.$

Kamei and Fujii \cite{FK1}, \cite{FK2} defined the \textit{relative operator
entropy} $S\left( A|B\right) ,$ for positive invertible operators $A$ and $B,
$ which is a relative version of the operator entropy considered by
Nakamura-Umegaki \cite{NU}. 

\section{Main Results}

We start to the following identity of interest:

\begin{lemma}
\label{l.2.1}Assume that the function $f:(0,\infty )\rightarrow \mathbb{R}$
is operator monotone in $(0,\infty )$ and has the representation (\ref{e.1}%
). Then for all $U,$ $V>0$ we have%
\begin{align}
f\left( V\right) -f\left( U\right) & =b\left( V-U\right)  \label{e.2.1} \\
& +\int_{0}^{\infty }\lambda ^{2}\left[ \int_{0}^{1}\left( \left( 1-t\right)
U+tV+\lambda \right) ^{-1}\right.  \notag \\
& \left. \times \left( V-U\right) \left( \left( 1-t\right) U+tV+\lambda
\right) ^{-1}dt\right] dw\left( \lambda \right) .  \notag
\end{align}
\end{lemma}

\begin{proof}
Since the function $f:(0,\infty )\rightarrow \mathbb{R}$ is operator
monotone in $(0,\infty )$ and has the representation (\ref{e.1}), then for $%
U,$ $V>0$ we have the representation%
\begin{equation}
f\left( V\right) -f\left( U\right) =b\left( V-U\right) +\int_{0}^{\infty
}\lambda \left[ V\left( V+\lambda \right) ^{-1}-U\left( U+\lambda \right)
^{-1}\right] dw\left( \lambda \right) .  \label{e.2.2}
\end{equation}%
Observe that for $\lambda >0$%
\begin{align*}
& V\left( V+\lambda \right) ^{-1}-U\left( U+\lambda \right) ^{-1} \\
& =\left( V+\lambda -\lambda \right) \left( V+\lambda \right) ^{-1}-\left(
U+\lambda -\lambda \right) \left( U+\lambda \right) ^{-1} \\
& =\left( V+\lambda \right) \left( V+\lambda \right) ^{-1}-\lambda \left(
V+\lambda \right) ^{-1}-\left( U+\lambda \right) \left( U+\lambda \right)
^{-1}+\lambda \left( U+\lambda \right) ^{-1} \\
& =\lambda \left[ \left( U+\lambda \right) ^{-1}-\left( V+\lambda \right)
^{-1}\right] .
\end{align*}%
Therefore, (\ref{e.2.2}) becomes, see also \cite{Fu1} 
\begin{equation}
f\left( V\right) -f\left( U\right) =b\left( V-U\right) +\int_{0}^{\infty
}\lambda ^{2}\left[ \left( U+\lambda \right) ^{-1}-\left( V+\lambda \right)
^{-1}\right] dw\left( \lambda \right) .  \label{e.2.2.a}
\end{equation}%
Let $T,$ $S>0.$ The function $f\left( t\right) =-t^{-1}$ is operator
monotonic on $\left( 0,\infty \right) $, operator G\^{a}teaux differentiable
and the G\^{a}teaux derivative is given by 
\begin{equation}
\nabla f_{T}\left( S\right) :=\lim_{t\rightarrow 0}\left[ \frac{f\left(
T+tS\right) -f\left( T\right) }{t}\right] =T^{-1}ST^{-1}  \label{e.2.3}
\end{equation}%
for $T,$ $S>0.$

Consider the continuous function $f$ defined on an interval $I$ for which
the corresponding operator function is G\^{a}teaux differentiable and for $%
C, $ $D$ selfadjoint operators with spectra in $I$ we consider the auxiliary
function defined on $\left[ 0,1\right] $ by 
\begin{equation*}
f_{C,D}\left( t\right) =f\left( \left( 1-t\right) C+tD\right) ,\text{ }t\in %
\left[ 0,1\right] .
\end{equation*}%
If $f_{C,D}$ is G\^{a}teaux differentiable on the segment $\left[ C,D\right]
:=\left\{ \left( 1-t\right) C+tD,\text{ }t\in \left[ 0,1\right] \right\} ,$
then we have, by the properties of the Bochner integral, that%
\begin{equation}
f\left( D\right) -f\left( C\right) =\int_{0}^{1}\frac{d}{dt}\left(
f_{C,D}\left( t\right) \right) dt=\int_{0}^{1}\nabla f_{\left( 1-t\right)
C+tD}\left( D-C\right) dt.  \label{e.2.4}
\end{equation}%
If we write this equality for the function $f\left( t\right) =-t^{-1}$ and $%
C,$ $D>0,$ then we get the representation 
\begin{equation}
C^{-1}-D^{-1}=\int_{0}^{1}\left( \left( 1-t\right) C+tD\right) ^{-1}\left(
D-C\right) \left( \left( 1-t\right) C+tD\right) ^{-1}dt.  \label{e.2.5}
\end{equation}%
Now, if we replace in (\ref{e.2.5}) $C=U+\lambda $ and $D=V+\lambda $ for $%
\lambda >0,$ then 
\begin{align}
& \left( U+\lambda \right) ^{-1}-\left( V+\lambda \right) ^{-1}
\label{e.2.6} \\
& =\int_{0}^{1}\left( \left( 1-t\right) U+tV+\lambda \right) ^{-1}\left(
V-U\right) \left( \left( 1-t\right) U+tV+\lambda \right) ^{-1}dt.  \notag
\end{align}%
By the representation (\ref{e.2.2.a}), we derive (\ref{e.2.1}).
\end{proof}

\begin{theorem}
\label{t.2.1}Assume that the function $f:(0,\infty )\rightarrow \mathbb{R}$
is operator monotone in $(0,\infty )$ and has the representation (\ref{e.1}%
). Then for all $A,$ $B,$ $P>0$ we have%
\begin{align}
& \mathcal{P}_{f}\left( B,P\right) -\mathcal{P}_{f}\left( A,P\right)
\label{e.2.7} \\
& =b\left( B-A\right) +\int_{0}^{\infty }\lambda ^{2}\left[
\int_{0}^{1}P\left( \left( 1-t\right) A+tB+\lambda P\right) ^{-1}\left(
B-A\right) \right.  \notag \\
& \left. \times \left( \left( 1-t\right) A+tB+\lambda P\right) ^{-1}Pdt 
\right] dw\left( \lambda \right) .  \notag
\end{align}
\end{theorem}

\begin{proof}
If we take $V=P^{-1/2}BP^{-1/2}$ and $U=P^{-1/2}AP^{-1/2}$ in (\ref{e.2.1}),
then we get 
\begin{align}
& f\left( P^{-1/2}BP^{-1/2}\right) -f\left( P^{-1/2}AP^{-1/2}\right) 
\label{e.2.8} \\
& =b\left( P^{-1/2}BP^{-1/2}-P^{-1/2}AP^{-1/2}\right)   \notag \\
& +\int_{0}^{\infty }\lambda ^{2}\left[ \int_{0}^{1}\left( \left( 1-t\right)
P^{-1/2}AP^{-1/2}+tP^{-1/2}BP^{-1/2}+\lambda \right) ^{-1}\right.   \notag \\
& \times \left( P^{-1/2}BP^{-1/2}-P^{-1/2}AP^{-1/2}\right)   \notag \\
& \left. \times \left( \left( 1-t\right)
P^{-1/2}AP^{-1/2}+tP^{-1/2}BP^{-1/2}+\lambda \right) ^{-1}dt\right] dw\left(
\lambda \right) .  \notag
\end{align}%
Observe that 
\begin{equation*}
P^{-1/2}BP^{-1/2}-P^{-1/2}AP^{-1/2}=P^{-1/2}\left( B-A\right) P^{-1/2},
\end{equation*}%
and%
\begin{equation*}
\left( 1-t\right) P^{-1/2}AP^{-1/2}+tP^{-1/2}BP^{-1/2}+\lambda
=P^{-1/2}\left( \left( 1-t\right) A+tB+\lambda P\right) P^{-1/2},
\end{equation*}%
which gives%
\begin{align*}
& \left( \left( 1-t\right) P^{-1/2}AP^{-1/2}+tP^{-1/2}BP^{-1/2}+\lambda
\right) ^{-1} \\
& =P^{1/2}\left( \left( 1-t\right) A+tB+\lambda P\right) ^{-1}P^{1/2}
\end{align*}%
and by (\ref{e.2.8}),%
\begin{align}
& f\left( P^{-1/2}BP^{-1/2}\right) -f\left( P^{-1/2}AP^{-1/2}\right) 
\label{e.2.9} \\
& =bP^{-1/2}\left( B-A\right) P^{-1/2}  \notag \\
& +\int_{0}^{\infty }\lambda ^{2}\left[ \int_{0}^{1}P^{1/2}\left( \left(
1-t\right) A+tB+\lambda P\right) ^{-1}P^{1/2}P^{-1/2}\left( B-A\right)
P^{-1/2}\right.   \notag \\
& \left. \times P^{1/2}\left( \left( 1-t\right) A+tB+\lambda P\right)
^{-1}P^{1/2}dt\right] dw\left( \lambda \right)   \notag \\
& =bP^{-1/2}\left( B-A\right) P^{-1/2}  \notag \\
& +\int_{0}^{\infty }\lambda ^{2}\left[ \int_{0}^{1}P^{1/2}\left( \left(
1-t\right) A+tB+\lambda P\right) ^{-1}\left( B-A\right) \right.   \notag \\
& \left. \times \left( \left( 1-t\right) A+tB+\lambda P\right) ^{-1}P^{1/2}dt
\right] dw\left( \lambda \right) .  \notag
\end{align}%
If we multiply both sides of (\ref{e.2.9}) by $P^{1/2}$ we obtain the
desired identity (\ref{e.2.7}).
\end{proof}

\begin{lemma}
\label{l.2.2}Assume that the function $f:(0,\infty )\rightarrow \mathbb{R}$
is operator monotone in $(0,\infty )$ and has the representation (\ref{e.1}%
). Then for all $U,$ $V>0$ we have%
\begin{align}
& \tilde{f}\left( V\right) -\tilde{f}\left( U\right)  \label{e.2.10} \\
& =a\left( V-U\right) +\int_{0}^{\infty }\lambda \left( \int_{0}^{1}\left(
1+\lambda \left[ \left( 1-t\right) U+tV\right] \right) ^{-1}\right.  \notag
\\
& \left. \times \left( V-U\right) \left( 1+\lambda \left[ \left( 1-t\right)
U+tV\right] \right) ^{-1}dt\right) dw\left( \lambda \right) .  \notag
\end{align}
\end{lemma}

\begin{proof}
From (\ref{e.1}) we have 
\begin{equation*}
f\left( t\right) =a+bt+t\int_{0}^{\infty }\frac{\lambda }{t+\lambda }%
dw\left( \lambda \right) ,\text{ }t>0.
\end{equation*}%
If we put $\frac{1}{t}$ instead of $t$ we get 
\begin{eqnarray*}
f\left( \frac{1}{t}\right) &=&a+b\frac{1}{t}+\frac{1}{t}\int_{0}^{\infty }%
\frac{\lambda }{\frac{1}{t}+\lambda }dw\left( \lambda \right) \\
&=&a+b\frac{1}{t}+\frac{1}{t}\int_{0}^{\infty }\frac{t\lambda }{1+t\lambda }%
dw\left( \lambda \right)
\end{eqnarray*}%
and by multiplication with $t>0,$ we get 
\begin{equation*}
\tilde{f}\left( t\right) =b+ta+\int_{0}^{\infty }\frac{t\lambda }{1+t\lambda 
}dw\left( \lambda \right) =b+ta+\int_{0}^{\infty }\left( 1-\frac{1}{%
1+t\lambda }\right) dw\left( \lambda \right) .
\end{equation*}%
Therefore 
\begin{equation}
\tilde{f}\left( V\right) -\tilde{f}\left( U\right) =a\left( V-U\right)
+\int_{0}^{\infty }\left[ \left( 1+U\lambda \right) ^{-1}-\left( 1+V\lambda
\right) ^{-1}\right] dw\left( \lambda \right) .  \label{e.2.10.a}
\end{equation}

From (\ref{e.2.5}) we get 
\begin{align}
& \left( 1+U\lambda \right) ^{-1}-\left( 1+V\lambda \right) ^{-1}
\label{e.2.11} \\
& =\int_{0}^{1}\left( \left( 1-t\right) \left( 1+U\lambda \right) +t\left(
1+V\lambda \right) \right) ^{-1}\left( \left( 1+V\lambda \right) -\left(
1+U\lambda \right) \right)  \notag \\
& \times \left( \left( 1-t\right) \left( 1+U\lambda \right) +t\left(
1+V\lambda \right) \right) ^{-1}dt  \notag \\
& =\int_{0}^{1}\lambda \left( 1+\lambda \left[ \left( 1-t\right) U+tV\right]
\right) ^{-1}\left( V-U\right) \left( 1+\lambda \left[ \left( 1-t\right) U+tV%
\right] \right) ^{-1}dt.  \notag
\end{align}%
Therefore, by (\ref{e.2.10.a}) we get 
\begin{align}
& \tilde{f}\left( V\right) -\tilde{f}\left( U\right)  \label{e.2.12} \\
& =a\left( V-U\right) +\int_{0}^{\infty }\left[ \left( 1+U\lambda \right)
^{-1}-\left( 1+V\lambda \right) ^{-1}\right] dw\left( \lambda \right)  \notag
\\
& =a\left( V-U\right) +\int_{0}^{\infty }\lambda \left( \int_{0}^{1}\left(
1+\lambda \left[ \left( 1-t\right) U+tV\right] \right) ^{-1}\right.  \notag
\\
& \left. \times \left( V-U\right) \left( 1+\lambda \left[ \left( 1-t\right)
U+tV\right] \right) ^{-1}dt\right) dw\left( \lambda \right)  \notag
\end{align}%
and the identity (\ref{e.2.10}) is proved.
\end{proof}

\begin{theorem}
\label{t.2.2}Assume that the function $f:(0,\infty )\rightarrow \mathbb{R}$
is operator monotone in $(0,\infty )$ and has the representation (\ref{e.1}%
). Then for all $C,$ $D,$ $Q>0$ we have%
\begin{align}
& \mathcal{P}_{\tilde{f}}\left( D,Q\right) -\mathcal{P}_{\tilde{f}}\left(
C,Q\right)  \label{e.2.13} \\
& =a\left( D-C\right) +\int_{0}^{\infty }\lambda \left( \int_{0}^{1}Q\left[
\left( Q+\lambda \left[ \left( 1-t\right) C+tD\right] \right) \right]
^{-1}\left( D-C\right) \right.  \notag \\
& \left. \times \left[ \left( Q+\lambda \left[ \left( 1-t\right) C+tD\right]
\right) \right] ^{-1}Qdt\right) dw\left( \lambda \right) .  \notag
\end{align}
\end{theorem}

\begin{proof}
If we take $V=Q^{-1/2}DQ^{-1/2}$ and $U=Q^{-1/2}CQ^{-1/2}$ in (\ref{e.2.10}%
), then we get%
\begin{align}
& \tilde{f}\left( Q^{-1/2}DQ^{-1/2}\right) -\tilde{f}\left(
Q^{-1/2}CQ^{-1/2}\right)   \label{e.2.14} \\
& =a\left( Q^{-1/2}DQ^{-1/2}-Q^{-1/2}CQ^{-1/2}\right)   \notag \\
& +\int_{0}^{\infty }\lambda \left( \int_{0}^{1}\left( 1+\lambda \left[
\left( 1-t\right) Q^{-1/2}CQ^{-1/2}+tQ^{-1/2}DQ^{-1/2}\right] \right)
^{-1}\right.   \notag \\
& \times \left( Q^{-1/2}DQ^{-1/2}-Q^{-1/2}CQ^{-1/2}\right)   \notag \\
& \times \left. \left( 1+\lambda \left[ \left( 1-t\right)
Q^{-1/2}CQ^{-1/2}+tQ^{-1/2}DQ^{-1/2}\right] \right) ^{-1}dt\right) dw\left(
\lambda \right)   \notag
\end{align}%
\begin{align*}
& =aQ^{-1/2}\left( D-C\right) Q^{-1/2} \\
& +\int_{0}^{\infty }\lambda \left( \int_{0}^{1}\left[ Q^{-1/2}\left(
Q+\lambda \left[ \left( 1-t\right) C+tD\right] \right) \right] ^{-1}\right. 
\\
& \left. \times Q^{-1/2}\left( D-C\right) Q^{-1/2}\left[ Q^{-1/2}\left(
Q+\lambda \left[ \left( 1-t\right) C+tD\right] \right) Q^{-1/2}\right]
^{-1}dt\right) dw\left( \lambda \right)  \\
& =aQ^{-1/2}\left( D-C\right) Q^{-1/2} \\
& +\int_{0}^{\infty }\lambda \left( \int_{0}^{1}Q^{1/2}\left[ \left(
Q+\lambda \left[ \left( 1-t\right) C+tD\right] \right) \right] ^{-1}\left(
D-C\right) \right.  \\
& \left. \times \left[ \left( Q+\lambda \left[ \left( 1-t\right) C+tD\right]
\right) \right] ^{-1}Q^{1/2}dt\right) dw\left( \lambda \right) .
\end{align*}%
If we multiply both sides by $Q^{1/2}$ we get the desired result (\ref%
{e.2.13}).
\end{proof}

\begin{corollary}
\label{c.2.1}Assume that the function $f:(0,\infty )\rightarrow \mathbb{R}$
is operator monotone in $(0,\infty )$ and has the representation (\ref{e.1}%
). Then for all $C,$ $D,$ $Q>0$ we have%
\begin{align}
& \mathcal{P}_{f}\left( Q,D\right) -\mathcal{P}_{f}\left( Q,C\right)
\label{e.2.15} \\
& =a\left( D-C\right) +\int_{0}^{\infty }\lambda \left( \int_{0}^{1}Q\left[
\left( Q+\lambda \left[ \left( 1-t\right) C+tD\right] \right) \right]
^{-1}\left( D-C\right) \right.  \notag \\
& \left. \times \left[ \left( Q+\lambda \left[ \left( 1-t\right) C+tD\right]
\right) \right] ^{-1}Qdt\right) dw\left( \lambda \right) .  \notag
\end{align}
\end{corollary}

We also have:

\begin{corollary}
\label{c.2.2}Assume that the function $f:(0,\infty )\rightarrow \mathbb{R}$
is operator monotone in $(0,\infty )$ and has the representation (\ref{e.1}%
). Then for all $A,$ $B,$ $C,$ $D>0$ we have%
\begin{align}
& \mathcal{P}_{f}\left( A,B\right) -\mathcal{P}_{f}\left( C,D\right)
\label{e.2.15.b} \\
& =b\left( A-C\right) +a\left( B-D\right)  \notag \\
& +\int_{0}^{\infty }\lambda ^{2}\left[ \int_{0}^{1}B\left( \left(
1-t\right) C+tA+\lambda B\right) ^{-1}\left( A-C\right) \right.  \notag \\
& \left. \times \left( \left( 1-t\right) C+tA+\lambda B\right) ^{-1}Bdt 
\right] dw\left( \lambda \right)  \notag \\
& +\int_{0}^{\infty }\lambda \left( \int_{0}^{1}C\left[ \left( C+\lambda %
\left[ \left( 1-t\right) D+tB\right] \right) \right] ^{-1}\left( B-D\right)
\right.  \notag \\
& \left. \times \left[ \left( C+\lambda \left[ \left( 1-t\right) D+tB\right]
\right) \right] ^{-1}Cdt\right) dw\left( \lambda \right) .  \notag
\end{align}
\end{corollary}

\begin{proof}
Observe that 
\begin{equation}
\mathcal{P}_{f}\left( A,B\right) -\mathcal{P}_{f}\left( C,D\right) =\mathcal{%
P}_{f}\left( A,B\right) -\mathcal{P}_{f}\left( C,B\right) +\mathcal{P}%
_{f}\left( C,B\right) -\mathcal{P}_{f}\left( C,D\right) .  \label{e.2.15.b1}
\end{equation}%
Since, by (\ref{e.2.7}), 
\begin{align}
& \mathcal{P}_{f}\left( A,B\right) -\mathcal{P}_{f}\left( C,B\right) 
\label{e.2.15.c} \\
& =b\left( A-C\right) +\int_{0}^{\infty }\lambda ^{2}\left[
\int_{0}^{1}B\left( \left( 1-t\right) C+tA+\lambda B\right) ^{-1}\left(
A-C\right) \right.   \notag \\
& \left. \times \left( \left( 1-t\right) C+tA+\lambda B\right) ^{-1}Bdt
\right] dw\left( \lambda \right)   \notag
\end{align}%
and by (\ref{e.2.15}), 
\begin{align}
& \mathcal{P}_{f}\left( C,B\right) -\mathcal{P}_{f}\left( C,D\right) 
\label{e.2.15.d} \\
& =a\left( B-D\right) +\int_{0}^{\infty }\lambda \left( \int_{0}^{1}C\left[
\left( C+\lambda \left[ \left( 1-t\right) D+tB\right] \right) \right]
^{-1}\left( B-D\right) \right.   \notag \\
& \left. \times \left[ \left( C+\lambda \left[ \left( 1-t\right) D+tB\right]
\right) \right] ^{-1}Cdt\right) dw\left( \lambda \right) ,  \notag
\end{align}%
hence by (\ref{e.2.15.b1})-(\ref{e.2.15.d}) we obtain (\ref{e.2.15.b}).
\end{proof}

As a natural consequence of the above representations, we derive the
following inequalities:

\begin{theorem}
\label{t.2.3}Assume that the function $f:(0,\infty )\rightarrow \mathbb{R}$
is operator monotone in $(0,\infty )$ and has the representation (\ref{e.1}%
). If $B\geq A>0$ and $P>0,$ then 
\begin{equation}
\mathcal{P}_{f}\left( B,P\right) -\mathcal{P}_{f}\left( A,P\right) \geq
b\left( B-A\right) \geq 0  \label{e.2.16}
\end{equation}%
and 
\begin{equation}
\mathcal{P}_{f}\left( P,B\right) -\mathcal{P}_{f}\left( P,A\right) \geq
a\left( B-A\right) .  \label{e.2.17}
\end{equation}

If $A\geq C>0$ and $B\geq D>0,$ then 
\begin{equation}
\mathcal{P}_{f}\left( A,B\right) -\mathcal{P}_{f}\left( C,D\right) \geq
b\left( A-C\right) +a\left( B-D\right) .  \label{e.2.17.a}
\end{equation}
\end{theorem}

\begin{proof}
If $B-A\geq 0,$ then by multiplying both sides by $\left( \left( 1-t\right)
A+tB+\lambda P\right) ^{-1}$ for $t\in \left[ 0,1\right] $ and $\lambda \geq
0$ we get%
\begin{equation*}
\left( \left( 1-t\right) A+tB+\lambda P\right) ^{-1}\left( B-A\right) \left(
\left( 1-t\right) A+tB+\lambda P\right) ^{-1}\geq 0.
\end{equation*}%
Also by multiplying both sides by $P>0,$ we get%
\begin{equation*}
P\left( \left( 1-t\right) A+tB+\lambda P\right) ^{-1}\left( B-A\right)
\left( \left( 1-t\right) A+tB+\lambda P\right) ^{-1}P\geq 0,
\end{equation*}%
for $t\in \left[ 0,1\right] $ and $\lambda \geq 0.$

If we multiply this inequality by $\lambda ^{2}$ integrate over $t\in \left[
0,1\right] $ and over the measure $w\left( \lambda \right) $ on $[0,\infty )$
we get 
\begin{align*}
& \int_{0}^{\infty }\lambda ^{2}\left[ \int_{0}^{1}P\left( \left( 1-t\right)
A+tB+\lambda P\right) ^{-1}\left( B-A\right) \right. \\
& \left. \times \left( \left( 1-t\right) A+tB+\lambda P\right) ^{-1}Pdt 
\right] dw\left( \lambda \right) \geq 0
\end{align*}%
and by representation (\ref{e.2.7}) we deduce (\ref{e.2.16}).

The inequality (\ref{e.2.17}) follows in a similar way by (\ref{e.2.15}).
The inequality (\ref{e.2.17.a}) follows by the representation (\ref{e.2.15.a}%
).
\end{proof}

\begin{remark}
\label{r.2.1}If $f:(0,\infty )\rightarrow \mathbb{R}$ is operator monotone
in $(0,\infty )$ and nonnegative, then in representation (\ref{e.1}) the
parameter $a$ must be nonnegative and in this situation we have%
\begin{equation}
\mathcal{P}_{f}\left( P,B\right) -\mathcal{P}_{f}\left( P,A\right) \geq
a\left( B-A\right) \geq 0,  \label{e.2.18}
\end{equation}%
if $B\geq A>0$ and $P>0.$

If $f$ is defined on $[0,\infty ),$ then we can take $a=f\left( 0\right) $
in (\ref{e.2.17}) and (\ref{e.2.18}). If the parameters $a$ and $b$ are
positive in representation (\ref{e.1}), then the inequality (\ref{e.2.17.a})
improves (\ref{e.3.2}).
\end{remark}

\section{Some Examples of Interest}

We also have identities for the \textit{weighted operator geometric mean:}

\begin{proposition}
\label{p.3.1}For all $A,$ $B,$ $P>0$ and $r\in (0,1]$ we have%
\begin{align}
& P\sharp _{r}B-P\sharp _{r}A  \label{e.2.9.a} \\
& =\frac{\sin \left( r\pi \right) }{\pi }\int_{0}^{\infty }\lambda ^{r+1}%
\left[ \int_{0}^{1}P\left( \left( 1-t\right) A+tB+\lambda P\right)
^{-1}\left( B-A\right) \right.   \notag \\
& \left. \times \left( \left( 1-t\right) A+tB+\lambda P\right) ^{-1}Pdt
\right] d\lambda .  \notag
\end{align}
\end{proposition}

The proof follows by (\ref{e.2.7}) and (\ref{e.0}) for the measure $dw\left(
\lambda \right) =\frac{\sin \left( r\pi \right) }{\pi }\lambda
^{r-1}d\lambda .$

The dual case follows by (\ref{e.2.15}) and (\ref{e.0}).

\begin{proposition}
\label{p.3.2}For all $C,$ $D,$ $Q>0$ and $r\in (0,1]$ we have%
\begin{align}
& D\sharp _{r}Q-C\sharp _{r}Q  \label{e.2.15.a} \\
& =\frac{\sin \left( r\pi \right) }{\pi }\int_{0}^{\infty }\lambda
^{r}\left( \int_{0}^{1}Q\left[ \left( Q+\lambda \left[ \left( 1-t\right) C+tD%
\right] \right) \right] ^{-1}\left( D-C\right) \right.  \notag \\
& \left. \times \left[ \left( Q+\lambda \left[ \left( 1-t\right) C+tD\right]
\right) \right] ^{-1}Qdt\right) d\lambda .  \notag
\end{align}
\end{proposition}

The following identity for the logarithmic function also holds:

\begin{lemma}
\label{l.2.3}For all $U,$ $V>0$ we have the identity:%
\begin{multline}
\ln V-\ln U  \label{e.2.19} \\
=\int_{0}^{\infty }\left( \int_{0}^{1}\left( \lambda +\left( 1-t\right)
U+tV\right) ^{-1}\left( V-U\right) \left( \lambda +\left( 1-t\right)
U+tV\right) ^{-1}dt\right) d\lambda .
\end{multline}
\end{lemma}

\begin{proof}
We have from the representation of logarithm (\ref{e.2}) that 
\begin{equation}
\ln V-\ln U=\int_{0}^{\infty }\frac{1}{\lambda +1}\left[ \left( V-1\right)
\left( \lambda +V\right) ^{-1}-\left( U-1\right) \left( \lambda +U\right)
^{-1}\right] d\lambda  \label{e.2.21}
\end{equation}%
for $U,$ $V>0.$

Since 
\begin{align*}
& \left( V-1\right) \left( \lambda +V\right) ^{-1}-\left( U-1\right) \left(
\lambda +U\right) ^{-1} \\
& =V\left( \lambda +V\right) ^{-1}-U\left( \lambda +U\right) ^{-1}-\left(
\left( \lambda +V\right) ^{-1}-\left( \lambda +U\right) ^{-1}\right)
\end{align*}%
and%
\begin{align*}
& V\left( \lambda +V\right) ^{-1}-U\left( \lambda +U\right) ^{-1} \\
& =\left( V+\lambda -\lambda \right) \left( \lambda +V\right) ^{-1}-\left(
U+\lambda -\lambda \right) \left( \lambda +U\right) ^{-1} \\
& =1-\lambda \left( \lambda +V\right) ^{-1}-1+\lambda \left( \lambda
+U\right) ^{-1}=\lambda \left( \lambda +U\right) ^{-1}-\lambda \left(
\lambda +V\right) ^{-1},
\end{align*}%
hence 
\begin{align*}
& \left( V-1\right) \left( \lambda +V\right) ^{-1}-\left( U-1\right) \left(
\lambda +U\right) ^{-1} \\
& =\lambda \left( \lambda +U\right) ^{-1}-\lambda \left( \lambda +V\right)
^{-1}-\left( \left( \lambda +V\right) ^{-1}-\left( \lambda +U\right)
^{-1}\right) \\
& =\left( \lambda +1\right) \left[ \left( \lambda +U\right) ^{-1}-\left(
\lambda +V\right) ^{-1}\right]
\end{align*}%
and by (\ref{e.2.21}) we get 
\begin{equation}
\ln V-\ln U=\int_{0}^{\infty }\left[ \left( \lambda +U\right) ^{-1}-\left(
\lambda +V\right) ^{-1}\right] d\lambda .  \label{e.2.22}
\end{equation}

Since, by (\ref{e.2.5}) we have 
\begin{align}
& \left( \lambda +U\right) ^{-1}-\left( \lambda +V\right) ^{-1}
\label{e.2.23} \\
& =\int_{0}^{1}\left( \lambda +\left( 1-t\right) U+tV\right) ^{-1}\left(
V-U\right) \left( \lambda +\left( 1-t\right) U+tV\right) ^{-1}dt,  \notag
\end{align}%
for all $\lambda \geq 0,$ hence by (\ref{e.2.22}) and (\ref{e.2.23}) we get (%
\ref{e.2.19}).
\end{proof}

\begin{theorem}
\label{t.2.4}For all $A,$ $B,$ $P>0$ we have%
\begin{align}
S\left( P|B\right) -S\left( P|A\right) & =\int_{0}^{\infty }\left[
\int_{0}^{1}P\left( \left( 1-t\right) A+tB+\lambda P\right) ^{-1}\left(
B-A\right) \right.   \label{e.2.24} \\
& \left. \times \left( \left( 1-t\right) A+tB+\lambda P\right) ^{-1}Pdt
\right] d\lambda .  \notag
\end{align}
\end{theorem}

\begin{proof}
Follows by Lemma \ref{l.2.3} by taking $V=P^{-1/2}BP^{-1/2}$ and $%
U=P^{-1/2}AP^{-1/2}$ and multiplying both sides by $P^{1/2}.$
\end{proof}

\end{document}